\documentclass{article}
\usepackage{amsmath}
\usepackage{amsfonts}
\usepackage{amsthm}
\usepackage{amssymb}
\usepackage{color}

\newtheorem{theorem}{Theorem}[section]
\newtheorem{lemma}[theorem]{Lemma}
\newtheorem{proposition}[theorem]{Proposition}

\newtheorem{problem}[theorem]{Problem}

\newtheorem{corollary}[theorem]{Corollary}

\theoremstyle{definition}
\newtheorem{definition}[theorem]{Definition}
\theoremstyle{remark}
\newtheorem{remark}[theorem]{Remark}

\newcommand\remove[1]{}

\def\mju{\mathcal{U}}

\def\f2{\mathbb{F}_2}

\newcommand{\ep}{\varepsilon}
\newcommand{\lin}{{\rm lin}\hskip0.02cm}

\begin{document}

\title{Distortion in the finite determination result for embeddings of locally finite metric spaces into Banach spaces}

\author{Sofiya Ostrovska and Mikhail~I.~Ostrovskii}

\date{\today}
\maketitle

\begin{large}

\begin{abstract}

Given a Banach space $X$ and a real number $\alpha\ge 1$, we
write: (1) $D(X)\le\alpha$ if, for any locally finite metric space
$A$, all finite subsets of which admit bilipschitz embeddings into
$X$ with distortions $\le C$, the space $A$ itself admits a
bilipschitz embedding into $X$ with distortion $\le \alpha\cdot
C$; (2) $D(X)=\alpha^+$ if,  for every $\ep>0$, the condition
$D(X)\le\alpha+\ep$ holds, while $D(X)\le\alpha$ does not; (3)
$D(X)\le \alpha^+$ if $D(X)=\alpha^+$ or $D(X)\le \alpha$. It is
known that $D(X)$ is bounded by a universal constant, but the
available estimates for this constant are rather large.

The following results have been proved in this work: (1)
$D((\oplus_{n=1}^\infty X_n)_p)\le 1^+$ for every nested family of
finite-dimensional Banach spaces $\{X_n\}_{n=1}^\infty$ and every
$1\le p\le \infty$. (2) $D((\oplus_{n=1}^\infty
\ell^\infty_n)_p)=1^+$ for $1<p<\infty$. (3) $D(X)\le 4^+$ for
every Banach space $X$ with no nontrivial cotype. Statement (3) is
a strengthening of the Baudier-Lancien result (2008).

\end{abstract}

\medskip



\section{Introduction}

The study of bilipschitz embeddings of metric spaces into Banach
spaces is a very active research area which has found many
applications, not only within Functional Analysis, but also in
Graph Theory, Group Theory, and Computer Science. see \cite{Lin02,
Mat02, Nao10, Ost13, WS11}.  This paper contributes to the study
of relations between the embeddability of an infinite metric space
and its finite pieces. Let us recollect some necessary notions.

\begin{definition} A metric space is called {\it locally finite}
if each ball of finite radius in it has finite cardinality.
\end{definition}

\begin{definition}\label{D:LipBilip}
{\rm (i)} Let $0\le C<\infty$. A map $f: (A,d_A)\to (Y,d_Y)$
between two metric spaces is called $C$-{\it Lipschitz} if
\[\forall u,v\in A\quad d_Y(f(u),f(v))\le Cd_A(u,v).\] A map $f$
is called {\it Lipschitz} if it is $C$-Lipschitz for some $0\le
C<\infty$.\medskip

{\rm (ii)} Let $1\le C<\infty$. A map, $f:A\to Y$, is called a
{\it $C$-bilipschitz embedding} if there exists $r>0$ such that
\begin{equation}\label{E:MapDist}\forall u,v\in A\quad rd_A(u,v)\le
d_Y(f(u),f(v))\le rCd_A(u,v).\end{equation} A  map $f$ is a {\it
bilipschitz embedding} if it is $C$-bilipschitz for some $1\le
C<\infty$. The smallest constant $C$ for which there exists $r>0$
such that \eqref{E:MapDist} is satisfied, is called the {\it
distortion} of $f$.
\end{definition}

We refer to \cite{LT77,Ost13} for unexplained terminology.\medskip

It has been known that the bilipschitz embeddability of locally
finite metric spaces into Banach spaces is finitely determined in
the following sense:

\begin{theorem}[\cite{Ost12}]\label{T:bilip} Let $A$ be a locally finite metric space whose
finite subsets admit bilipschitz embeddings with uniformly bounded
distortions into a Banach space $X$. Then, $A$ also admits a
bilipschitz embedding into $X$.
\end{theorem}

To elaborate more, the argument of \cite{Ost12} leads to a
stronger result which we state as Theorem \ref{T:FDwithD}. To
formulate Theorem \ref{T:FDwithD}, it is convenient to introduce
parameter $D(X)$ of a Banach space $X$. More specifically, given a
Banach space $X$ and a real number $\alpha\ge 1$, we write:

\begin{itemize}

\item  $D(X)\le\alpha$ if, for any locally finite metric space
$A$, all finite subsets of which admit bilipschitz embeddings into
$X$ with distortions $\le C$, the space $A$ itself admits a
bilipschitz embedding into $X$ with distortion $\le \alpha\cdot
C$;

\item  $D(X)=\alpha$ if $\alpha$ is the least number for which
$D(X)\le\alpha$;

\item  $D(X)=\alpha^+$ if,  for every $\ep>0$, the condition
$D(X)\le\alpha+\ep$ holds, while $D(X)\le\alpha$ does not;

\item $D(X)=\infty$ if $D(X)\le\alpha$ does not hold for
any $\alpha<\infty$.

\end{itemize}

Further, we use inequalities like $D(X)<\alpha^+$ and
$D(X)<\alpha$ with the natural meanings, for example
$D(X)<\alpha^+$ indicates that either $D(X)=\beta$ for some
$\beta\le\alpha$ or $D(X)=\beta^+$ for some $\beta<\alpha$.

\begin{theorem}[\cite{Ost12}]\label{T:FDwithD} There exists an absolute constant $D\in[1,\infty)$, such
that for an arbitrary Banach space $X$ the inequality $D(X)\le D$
holds.
\end{theorem}

In the proof of Theorem \ref{T:FDwithD} given in \cite{Ost12} as
well as in the proofs of its special cases obtained in
\cite{BL08}, \cite{Ost09}, and \cite{Bau12}, the values of $D$
implied by the argument are `large'. For example, Baudier and
Lancien in \cite{BL08} worked out the numerical estimate provided
by their proof and derived estimate $D(X)\le 216$ for Banach
spaces with no nontrivial cotype.
\medskip

On the other hand, it is known that for some Banach spaces $X$ the
value of $D(X)$ is significantly smaller. In order to present
relevant assertions, it is expedient to introduce the following
definition.

\begin{definition} It is said  that a Banach space $X$ satisfies
the {\it  condition} {\bf(U)} if each separable subset of an
arbitrary ultrapower of $X$ admits an isometric embedding into
$X$.\end{definition}

The fact stated below is well known and its proof follows
immediately from \cite[Proposition 2.21]{Ost13}:

\begin{proposition}\label{P:Ultra} If a Banach space $X$ satisfies  condition
{\bf (U)}, then  $D(X)=1$.
\end{proposition}

Further, the  next   result  due to Kalton and Lancien has to be
cited in the context of the present work.

\begin{theorem}[{\cite[Theorem 2.9]{KL08}}]\label{T:KLc_0} $D(c_0)=1^+$.
\end{theorem}

\begin{remark} Theorem 2.9 in \cite{KL08} is stated in terms of locally compact metric spaces. However, the corresponding lower bound is proved also for
locally finite metric spaces \cite[page~256]{KL08}, yielding
Theorem \ref{T:KLc_0}.
\end{remark}

The purport of this work is to find upper estimates for $D(X)$
which are significantly stronger than the estimates implied by the
proofs in \cite{BL08,Ost09,Bau12,Ost12}. Theorems \ref{T:Above},
\ref{T:Below}, \ref{T:BL08Impr}, and their corollaries constitute
the main results of the present paper.

Customarily, a family of finite-dimensional Banach spaces
$\{X_n\}_{n=1}^\infty$ is said to be {\it nested} if $X_n$ is a
proper subspace of $X_{n+1}$ for every $n\in\mathbb{N}$.

\begin{theorem}\label{T:Above}  Let $1\le
p<\infty$. If $\{X_n\}_{n=1}^\infty$ is a nested family of
finite-dimensional Banach spaces, then
$\displaystyle{D\left(\left(\oplus_{n=1}^\infty
X_n\right)_p\right)\le 1^+}$.
\end{theorem}

The main idea of our proofs of Theorems \ref{T:Above} and \ref{T:BL08Impr} is explained in Remark \ref{R:Idea}.

\begin{corollary}\label{C:Ell_p} If $1\le p<\infty$, then
$D(\ell_p)\le 1^+$.
\end{corollary}

\begin{remark} The problem of finiteness of $D(\ell_p)$, $p\ne 2,\infty$, was raised
by Marc Bourdon and published in \cite[Question 10.7]{NP11}. A
solution to this problem was found  in \cite{Bau12} and
\cite{Ost12}, but in both of these papers the bounds on
$D(\ell_p)$ are rather large numbers.
\end{remark}

In some cases, the inequality in Theorem \ref{T:Above} can be
reversed, as claimed by the forthcoming result:

\begin{theorem}\label{T:Below} Let $1<p<\infty$, then
$\displaystyle{D\left(\left(\oplus_{n=1}^\infty
\ell^\infty_n\right)_p\right)\ge 1^+}$.
\end{theorem}

Together with the pertinent  special case of Theorem \ref{T:Above}
this leads to:

\begin{corollary} Let $1<p<\infty$, then
$\displaystyle{D\left(\left(\oplus_{n=1}^\infty
\ell^\infty_n\right)_p\right)=1^+}$.
\end{corollary}

Our final goal is a significant improvement of the distortion
estimate obtained in \cite{BL08}. In this connection, the
following outcome has been reached:

\begin{theorem}\label{T:BL08Impr} Let $X$ be a Banach space with no nontrivial cotype.
Then $D(X)\le 4^+$.
\end{theorem}

\section{Proof of Theorem \ref{T:Above}}

Let $X=\left(\oplus_{n=1}^\infty X_n\right)_p$, $C\in [1,\infty)$,
and let $A$ be a locally finite metric space such that its finite
subsets admit embeddings into $X$ with distortion $\le C$. It has
to be proved that, for each $\ep>0$, there exists a bilipschitz
embedding of $A$ into $X$ with distortion $\le C+\ep$. By the
well-known fact (see \cite[Proposition 2.21]{Ost13}), such a space
$A$ admits a bilipschitz embedding with distortion $\le C$ into
any ultrapower of $X$. Thence, it is sufficient to show that, for
any $\ep>0$, every locally finite metric subspace $M$ of each
ultrapower $X^\mju$ admits a bilipschitz embedding into $X$ with
distortion $\le 1+\ep$. This can be accomplished by selecting an
arbitrary $\ep>0$ and finding a bilipschitz embedding of a locally
finite metric subspace $M$ of $X^\mju$ into $X$ with distortion
$\le 1+\varphi(\ep)$, where function $\varphi$ is such that
$\varphi(\ep)\downarrow 0$ as $\ep\downarrow 0$.
\medskip

Without loss of generality, one may assume that $0\in M$. Let
$\{R_n\}_{n=1}^\infty$ be an increasing sequence of positive real
numbers (we shall choose a sequence $\{R_n\}_{n=1}^\infty$ which
is suitable for our purposes later). Consider subsets $M_n$ of $M$
defined by
\[M_n=\{x\in M:~||x||\le R_n\}.\]
Since $M$ is a locally finite metric space, these sets are finite.
Therefore, by the definition of an ultrapower, there exist
bilipschitz embeddings of distortion $<1+\ep$ of these sets into
$X$. It follows immediately that, for each $n\in\mathbb{N}$, there
exists $t(n)\in\mathbb{N}$ such that $t(n+1)\ge t(n)$ and the
direct sum $\left(\oplus_{k=1}^{t(n)} X_k\right)_p$ admits a
bilipschitz embedding of $M_n$  with distortion $<1+\ep$. Apart
from that, since $X_n$, $n\in\mathbb{N}$, is a nested family of
spaces, this implies that $M_n$ admits a bilipschitz embedding
with distortion $< 1+\ep$ into the space
$Y_n:=\left(\oplus_{k=m(n-1)+1}^{m(n)} X_k\right)_p$, where
$m(0)=0$ and $m(n)=m(n-1)+t(n)$. It is easy to see that $Y_n$ is a
nested family of finite-dimensional Banach spaces and that
$X=\left(\oplus_{n=1}^\infty Y_n\right)_p$. We select and fix
embeddings $E_{n}:M_n\to Y_n$ with distortion $<(1+\ep)$. Without
loss of generality, it can be assumed that $E_n0=0$ and
\begin{equation}\label{Mn}
\forall x,y\in M_n\quad ||x-y||\le||E_nx-E_ny||<(1+\ep)||x-y||.
\end{equation}

\begin{remark}\label{R:Idea} Before we proceed, it seems beneficial to describe the main idea behind our proofs of Theorems \ref{T:Above} and
\ref{T:BL08Impr}. We have already introduced a sequence
$\{E_n\}_{n=1}^\infty$ of embeddings of balls in $M$ with
increasing radii into $X$. Now, what remains is to find a
low-distortion pasting technique for these maps. This is done by
rather complicated formulae, namely, \eqref{E:ci2}--\eqref{E:DefT}
and \eqref{E:ci3}--\eqref{E:DefTp>2}, which, in the case of
$\ell_2$-sums, reduce to what can be called an $\ep$-normalization
of the formula for the logarithmic spiral in the Euclidean plane:
$\gamma_\ep:(1,\infty)\to\mathbb{R}^2$,~~
$\gamma_\ep(t)=t(\cos(\ep\ln t),~\sin(\ep\ln t))$. The curve
$\gamma_\ep$ is a slight modification of the well-known example of
a quasi-geodesic in $\mathbb{R}^2$ which is far from geodesic, see
\cite[p.~4]{BS07}.

One can view this  pasting techniques as a transition from
$E_{2n}$ to $E_{2n+2}$ along $\ep$-normalized $\ell_p$-versions of
the logarithmic spiral. See \eqref{E:ci2}--\eqref{E:DefT} and
\eqref{E:ci3}--\eqref{E:DefTp>2}. The low-distortion estimates for
these embeddings are very close to the estimate which shows that
the map $\gamma_\ep$ has distortion $\le(1+\kappa(\ep))$ with
$(1+\kappa(\ep))\downarrow 1$ as $\ep\downarrow0$.
\end{remark}
\medskip

To continue the proof, we opt for an increasing sequence
$\{R_i\}_{i=1}^\infty$ of real numbers such that

\begin{equation}\label{E:R12}R_1=1,\end{equation}

\begin{equation}\label{E:REven2}\ep
\ln(R_{2i}/R_{2i-1})=\frac{\pi}2,\end{equation}

\begin{equation}\label{E:ROdd2}\displaystyle{\frac{R_{2i+1}}{R_{2i}}\ge\frac1{\ep}}.
\end{equation}
From this point on, we  are going to consider the cases $1\le p\le
2$ and $2<p<\infty$ separately, mostly because in the case $1\le
p\le 2$ much simpler formulae can be used.

\subsection{Spaces $\left(\oplus_{n=1}^\infty X_n\right)_p$, $1\le
p\le 2$}\label{S:Between1and2}

To construct an  embedding $T:M\to X$ with needful properties, we
employ the real-valued functions $c_{2i-1}$ and $s_{2i-1}$, $i\in
\mathbb{N}$  on  $M$ defined by:

\begin{equation}\label{E:ci2}
c_{2i-1}(x)=\begin{cases} \cos^{2/p}(\ep\ln(R_{2i-1}/R_{2i-1}))=1 &\hbox{ if }||x||\le R_{2i-1}\\
\cos^{2/p}(\ep\ln(||x||/R_{2i-1}))
&\hbox{ if } R_{2i-1}\le ||x||\le R_{2i}\\
\cos^{2/p}(\ep\ln(R_{2i}/R_{2i-1}))=0 &\hbox{ if }||x||\ge R_{2i}\\
\end{cases}
\end{equation}

\begin{equation}\label{E:si2}
s_{2i-1}(x)=\begin{cases} \sin^{2/p} (\ep\ln(R_{2i-1}/R_{2i-1}))=0 &\hbox{ if }||x||\le R_{2i-1}\\
\sin^{2/p} (\ep\ln(||x||/R_{2i-1}))
&\hbox{ if } R_{2i-1}\le ||x||\le R_{2i}\\
\sin^{2/p} (\ep\ln(R_{2i}/R_{2i-1}))=1 &\hbox{ if }||x||\ge R_{2i}\\
\end{cases}
\end{equation}
The equalities in the last lines of formulae \eqref{E:ci2} and
\eqref{E:si2} follow from \eqref{E:REven2}. Consider the map
$T:M\to X$ represented  by:
\begin{equation}\label{E:DefT}
Tx=\begin{cases} c_1(x) E_{2}x+s_1(x)E_{4}x &\hbox{if }x\in
M_{{3}}\\
c_3(x) E_{4}x+s_3(x)E_{6}x &\hbox{if }x\in
M_{{5}}\backslash M_{{3}}\\
\dots &\dots\\
c_{2i-1}(x) E_{{2i}}x+s_{2i-1}(x)E_{{2i+2}}x &\hbox{if }x\in
M_{{2i+1}}\backslash M_{{2i-1}}\\
\dots &\dots,\\
\end{cases}
\end{equation}
where we use the convention
that a product of $0$ and an undefined quantity is $0$.
Since $\left(c_{2i-1}(x)\right)^p+\left(s_{2i-1}(x)\right)^p=1$
for all $i$ and $x$, one  derives applying \eqref{Mn},
\eqref{E:DefT}, $E_n0=0$, and $X=\left(\oplus_{n=1}^\infty
Y_n\right)_p$, that
\begin{equation}\label{E:SameNorm2}
\forall x\in M\quad ||x||\le ||Tx||< (1+\ep)||x||.
\end{equation}
What is demanded now is an estimate of the form:
\begin{equation}\label{E:Desired}
\forall x,y\in M\quad (1-\psi(\ep))||x-y||\le ||Tx-Ty||<
(1+\xi(\ep))||x-y||,
\end{equation}
where functions $\psi$ and $\xi$ have positive values and are such that $\lim_{\ep\downarrow
0}\psi(\ep)=\lim_{\ep\downarrow 0}\xi(\ep)=0$.\medskip

Obviously, it suffices to consider the case $||y||\le ||x||$. The
simpler case $||y||\le{\ep}||x||$ creates no difficulty because if
this occurs, one obtains:
\begin{equation}\label{E:EpEst1}(1-{\ep})||x||\le
||x||-||y||\le
||x-y||\le||x||+||y||\le(1+{\ep})||x||\end{equation} and
\begin{equation}\label{E:EpEst}\begin{split} (1-\ep(1+\ep))||x||&\le
||x||-(1+\ep)||y||\le ||Tx||-||Ty||\\
&\le ||Tx-Ty||\le||Tx||+||Ty||\\&
\le(1+{\ep})||x||+(1+\ep)||y||\le
(1+\ep)^2||x||.\end{split}\end{equation} Combining
\eqref{E:EpEst1} and \eqref{E:EpEst}, we get
\begin{equation}\label{E:EpEstFin}
\frac{1-\ep(1+\ep)}{1+\ep}\,||x-y||\le ||Tx-Ty||\le
\frac{(1+\ep)^2}{1-\ep}\,||x-y||,
\end{equation}
which is an estimate of the required form
\eqref{E:Desired}.\medskip

As a next step, set $R_0=0$. By virtue of condition
\eqref{E:ROdd2} and inequality \eqref{E:EpEstFin}, it is enough to
consider the case where
\begin{equation}\label{E:Annulus} R_{2i-2}\le ||y||\le ||x||\le
R_{2i+1}, \quad i=1,2,\dots.\end{equation} It should be pointed
out that since functions $c_{2i-1}$ and $s_{2i-1}$ are constant on
intervals of the form $[R_{2j},R_{2j+1}]$, there are many trivial cases. Out of
the remaining ones we deal first with the case $R_{2i-1}\le ||y||\le
||x||\le R_{2i}$. \medskip

For simplicity of notation in the following calculations, it is
handy to  use $c$ for $c_{2i-1}$, $s$ for $s_{2i-1}$, $E$ for
$E_{{2i}}$, and $F$ for $E_{{2i+2}}$. With this in mind, one
has:
\begin{equation}\label{E:LpSum}\begin{split}&||Tx-Ty||^p=||c(x)Ex-c(y)Ey||^p+||s(x)Fx-s(y)Fy||^p\\
&~~=||c(x)(Ex-Ey)+(c(x)-c(y))Ey||^p\\&\qquad+||s(x)(Fx-Fy)+(s(x)-s(y))Fy||^p.\end{split}\end{equation}
Consider  each of the summands in the last line separately. To
begin with, the Mean Value Theorem yields:
\begin{equation}\label{E:MVTp<2}\begin{split}&c(x)-c(y)=\cos^{2/p}(\ep\ln(||x||/R_{2i-1}))-\cos^{2/p}(\ep\ln(||y||/R_{2i-1}))\\
&=\frac2p\,\cos^{\frac2p-1}(\ep\ln(\tau/R_{2i-1}))\cdot(-\sin(\ep\ln(\tau/R_{2i-1})))\cdot\ep\frac1{\tau}(||x||-||y||)
\end{split}\end{equation}
for some number $\tau$ satisfying $\tau\in(||y||,||x||)$. Now,
recall that $1\le p\le 2$ and hence $\frac2p-1\ge 0$. Therefore,
\begin{equation}\label{E:cxcy}
||(c(x)-c(y))Ey||\le\frac2p\cdot\ep\frac1{\tau}(||x||-||y||)\cdot(1+\ep)||y||\le
2\ep(1+\ep)||x-y||.
\end{equation}
Similarly, it can be demonstrated that
\begin{equation}\label{E:sxsy} ||(s(x)-s(y))Ey||\le 2\ep(1+\ep) ||x-y||.
\end{equation}
Inequalities \eqref{E:LpSum}, \eqref{E:cxcy}, and \eqref{E:sxsy}
lead to:
\begin{equation}\label{E:Fin_p}\begin{split}((\max\{c(x)-2\ep(1+\ep),&0\})^p+(\max\{s(x)-2\ep(1+\ep),0\})^p)||x-y||^p\\
&\le||Tx-Ty||^p\\&\le(1+\ep)^p((c(x)+2\ep)^p+(s(x)+2\ep)^p)||x-y||^p.\end{split}\end{equation}

Notice that
\[\lim_{\ep\downarrow
0}((\max\{c(x)-2\ep(1+\ep),0\})^p+(\max\{s(x)-2\ep(1+\ep),0\})^p)=1\]
and
\[\lim_{\ep\downarrow
0}(1+\ep)^p((c(x)+2\ep)^p+(s(x)+2\ep)^p)=1\] due to the fact that
$c^p(x)+s^p(x)=1$. Thus, inequality \eqref{E:Fin_p} provides the
desired estimate \eqref{E:Desired}.

To complete the proof, consider the case where
$||y||\in[R_{2i-2},R_{2i-1}]$ and $||x||\in[R_{2i-1},R_{2i}]$.
Then $c_{2i-1}(y)=\cos^{2/p} (\ep\ln(R_{2i-1}/R_{2i-1}))$, and,
therefore, proceeding as in \eqref{E:MVTp<2}
and as in the first inequality in \eqref{E:cxcy} we get
\[||(c(x)-c(y))Ey||\le\frac2p\cdot\ep\frac1{\tau}(||x||-R_{2i-1})\cdot(1+\ep)||y||\]
for some number $\tau\in(R_{2i-1},||x||)$. Hence
\[||(c(x)-c(y))Ey||\le
2\ep(1+\ep)||x-y|| \] in this case, too. Likewise, one can check
that \eqref{E:sxsy} holds as well. The other subcases of
\[R_{2i-2}\le||y||\le ||x||\le R_{2i+1}\]
can be treated in  the same manner.

\subsection{Spaces $\left(\oplus_{n=1}^\infty X_n\right)_p$, $p>2$}

The maps  used in the case $1\le p\le 2$ are not suitable for
$p>2$ because the power of cosine in \eqref{E:MVTp<2} becomes
negative and a nontrivial estimate does not come out in this way.
To get around this problem,  functions $c_{2i-1}$ and $s_{2i-1}$,
$i\in \mathbb{N}$ will be chosen differently.

We start by introducing the functions
$f_p:\left[0,\frac\pi2\right]\to \mathbb{R}$ and
$g_p:\left[0,\frac\pi2\right]\to \mathbb{R}$ by
\begin{equation}\label{E:f_p}
f_p(t)=\frac{\cos t}{(\cos^pt+\sin^pt)^{\frac1p}}, \qquad
g_p(t)=\frac{\sin t}{(\cos^pt+\sin^pt)^{\frac1p}}.
\end{equation}
It is clear that
\begin{equation}\label{E:f_p1}(f_p(t))^p+(g_p(t))^p=1.\end{equation}
Now, define $c_{2i-1}$ and $s_{2i-1}$, $i\in \mathbb{N}$, as
follows:

\begin{equation}\label{E:ci3}
c_{2i-1}(x)=\begin{cases} f_p (\ep\ln(R_{2i-1}/R_{2i-1}))=1 &\hbox{ if }||x||\le R_{2i-1}\\
f_p (\ep\ln(||x||/R_{2i-1}))
&\hbox{ if } R_{2i-1}\le ||x||\le R_{2i}\\
f_p (\ep\ln(R_{2i}/R_{2i-1}))=0 &\hbox{ if }||x||\ge R_{2i}\\
\end{cases}
\end{equation}

\begin{equation}\label{E:si3}
s_{2i-1}(x)=\begin{cases} g_p(\ep\ln(R_{2i-1}/R_{2i-1}))=0 &\hbox{ if }||x||\le R_{2i-1}\\
g_p (\ep\ln(||x||/R_{2i-1}))
&\hbox{ if } R_{2i-1}\le ||x||\le R_{2i}\\
g_p (\ep\ln(R_{2i}/R_{2i-1}))=1 &\hbox{ if }||x||\ge R_{2i}\\
\end{cases}
\end{equation}
The equalities in the last lines of formulae \eqref{E:ci3} and
\eqref{E:si3} can be derived from \eqref{E:REven2}. Similar to the
construction of the previous section, let us introduce the map
$T:M\to X$ by:
\begin{equation}\label{E:DefTp>2}
Tx=\begin{cases} c_1(x) E_{2}x+s_1(x)E_{4}x &\hbox{ if }x\in
M_{{3}}\\
c_3(x) E_{4}x+s_3(x)E_{6}x &\hbox{ if }x\in
M_{{5}}\backslash M_{{3}}\\
\dots &\dots\\
c_{2i-1}(x) E_{{2i}}x+s_{2i-1}(x)E_{{2i+2}}x &\hbox{ if }x\in
M_{{2i+1}}\backslash M_{{2i-1}}\\
\dots &\dots\\
\end{cases}
\end{equation}
In this equation $R_i$, $E_{i}$ and
$M_{i}$ have the same meaning as in our argument for $1\le p\le
2$.
The equation \eqref{E:f_p1} implies that
$\left(c_{2i-1}(x)\right)^p+\left(s_{2i-1}(x)\right)^p=1$ for all
$i$ and $x$. Therefore
\begin{equation}\label{E:SameNorm22}
\forall x\in M\quad ||x||\le ||Tx||\le (1+\ep)||x||.
\end{equation}
If $||y||\le{\ep}||x||$, the desired estimate \eqref{E:Desired} can be proved  in exactly the same way as in the
case $1\le p\le 2$. For the same reason as in the case $1\le p\le
2$, it suffices to consider the case where $R_{2i-1}\le ||y||\le
||x||\le R_{2i}$. For simplicity of notation in what follows
 we use $c$ for $c_{2i-1}$, $s$ for $s_{2i-1}$, $E$ for
$E_{{2i}}$, and $F$ for $E_{{2i+2}}$. Having said so, we
write:
\begin{equation}\label{E:Directp>2}\begin{split}&||Tx-Ty||^p=||c(x)Ex-c(y)Ey||^p+||s(x)Fx-s(y)Fy||^p\\
&~~=||c(x)(Ex-Ey)+(c(x)-c(y))Ey||^p\\&\qquad+||s(x)(Fx-Fy)+(s(x)-s(y))Fy||^p.\end{split}\end{equation}
Examine  each of the summands in the last line separately.
Notice that $c(x)-c(y)=F(||x||)-F(||y||)$, where
\[F(r)=\frac{G(r)}{B(r)},\]
\[G(r)=\cos(\ep\ln(r/R_{2i-1}))\]
\[B(r)=(\cos^p(\ep\ln(r/R_{2i-1}))+\sin^p(\ep\ln(r/R_{2i-1})))^{1/p}
\]
By the Mean Value Theorem,
\begin{equation}
F(||x||)-F(||y||)=\frac{G'(\tau)B(\tau)-G(\tau)B'(\tau)}{(B(\tau))^2}(||x||-||y||)
\end{equation}
for some $\tau\in(||y||,||x||)$. Obviously (recall that $p>2$),
\[2^{-\frac{p}2+1}\le \cos^pt+\sin^pt\le 1\]
and hence
\[2^{-\frac12+\frac1p}\le B(\tau)\le 1.\]
In addition,
\[G'(\tau)=-\sin(\ep\ln(\tau/R_{2i-1}))\ep\frac1{\tau}\]
whence
\[|G'(\tau)|\le\frac{\ep}{\tau}.\]
By plain calculations,
\[\begin{split}B'(\tau)&=\frac1p\,(B(\tau))^{1-p}\left(p\cos^{p-1}(\ep\ln(\tau/R_{2i-1}))\cdot(-\sin(\ep\ln(\tau/R_{2i-1})))\cdot\frac{\ep}{\tau}\right.
\\&\left.+p\sin^{p-1}(\ep\ln(\tau/R_{2i-1}))\cdot\cos(\ep\ln(\tau/R_{2i-1}))\cdot\frac{\ep}{\tau}\right),
\end{split}
\]
which implies:
\[|B'(\tau)|\le\left(2^{-\frac12+\frac1p}\right)^{1-p}\left(\frac{\ep}{\tau}+\frac{\ep}{\tau}\right).\]
Using  the obvious bound  $|G(\tau)|\le 1$, one
arrives at:
\[
\left|\frac{G'(\tau)B(\tau)-G(\tau)B'(\tau)}{(B(\tau))^2}\right|\le\frac{\frac{\ep}{\tau}+2^{\frac{(p-1)(p-2)}{2p}}\cdot
2\frac{\ep}{\tau}} {2^{2(\frac1p-\frac12)}}=C(p)\frac{\ep}{\tau},
\]
where $C(p)$ is some constant depending on $p$ only. Since
$\tau\in(||y||,||x||)$, it can be established that
\[ ||(c(x)-c(y))Ey||\le C(p)\frac\ep{\tau}(||x||-||y||)\cdot(1+\ep)||y||\le
\ep(1+\ep)C(p)||x-y||.
\]
Likewise, it can be shown that
\[ ||(s(x)-s(y))Ey||\le \ep(1+\ep)C(p)||x-y||.
\]
Combining the preceding inequalities with \eqref{E:Directp>2}, one concludes that the next
estimate is valid:
\begin{equation}\label{E:Fin_p>2}\begin{split}((\max\{c(x)&-\ep(1+\ep)C(p),0\})^p\\&\qquad+(\max\{s(x)-\ep(1+\ep)C(p),0\})^p)||x-y||^p\\
&\le||Tx-Ty||^p\\&\le(1+\ep)^p((c(x)+\ep C(p))^p+(s(x)+\ep
C(p))^p)||x-y||^p.\end{split}\end{equation} Clearly,
\eqref{E:f_p1} implies  that $c^p(x)+s^p(x)=1$, whence
\[\lim_{\ep\downarrow
0}((\max\{c(x)-\ep(1+\ep)C(p),0\})^p+(\max\{s(x)-\ep(1+\ep)C(p),0\})^p)=1\]
and
\[\lim_{\ep\downarrow
0}(1+\ep)^p((c(x)+\ep C(p))^p+(s(x)+\ep C(p))^p)=1.\] Thus, the
inequality \eqref{E:Fin_p>2} is of the desired type
\eqref{E:Desired}. $\hfill\Box$

\section{Proof of Theorem \ref{T:Below}}

\begin{proof} By the well-known observation of Fr\'echet
\cite[p.~161]{Fre10} (see also \cite[Proposition 1.17]{Ost13}),
all finite metric spaces admit isometric embeddings into
$X=\left(\oplus_{n=1}^\infty\ell_\infty^n\right)_p$. Therefore, to
prove Theorem \ref{T:Below}, a construction of a locally finite
metric space $A$ which is not isometric to a subset of $X$  (for $1<p<\infty$) is needed.
\medskip

The following notation for $X$ will be employed. Each element
$x\in X$ is a sequence $x=\{x_n\}_{n=1}^\infty$, where
$x_n\in\ell_\infty^n$. The norm of $x$ in $X$ will be denoted by
$||x||_X$. By the definition of direct sums one has:

\begin{equation}\label{E:NormInX}
||x||_X=\left(\sum_{n=1}^\infty||x_n||_\infty^p\right)^\frac1p,\end{equation}
where $||x_n||_\infty$ is the norm in $\ell_\infty^n$ (with slight
abuse of notation we use the same notation for all $n$). Denoting
the norm of $\ell_p$ by $||\cdot||_p$,  the right-hand side of
\eqref{E:NormInX} can be written as
$||\{||x_n||_\infty\}_{n=1}^\infty||_p$.

At this stage, some simple geometric properties of $X$ are needed.
Consider triples of points $x,y,z\in X$ satisfying
\begin{equation}\label{E:FlatTriple}
||x-z||_X=||x-y||_X+||y-z||_X. \end{equation} Let $x=\{x_n\}$,
$y=\{y_n\}$, $z=\{z_n\}$, where $x_n,y_n,z_n\in\ell_\infty^n$ are
the  components of $x$, $y$, and $z$, respectively.

\begin{lemma}\label{L:Multiple} For any triple $x,y,z\in X$ of pairwise distinct vectors satisfying  \eqref{E:FlatTriple}, the vector
$\{||x_n-y_n||_\infty\}_{n=1}^\infty\in\ell_p$ is a positive
multiple of $\{||y_n-z_n||_\infty\}_{n=1}^\infty\in\ell_p$.
\end{lemma}

\begin{proof} Assume the contrary. Recall that $1<p<\infty$. Using the fact that for $u,v\in\ell_p$ the inequality $||u+v||_p\le||u||_p+||v||_p$ is
strict if $u$ and $v$ are nonzero and are not positive multiples
of each other, one derives that the $\ell_p$-norm of the vector
$\{||x_n-y_n||_\infty+||y_n-z_n||_\infty\}_{n=1}^\infty$ is
strictly less than
\[\left\|\{||x_n-y_n||_\infty\}_{n=1}^\infty\right\|_p+||\{||y_n-z_n||_\infty\}_{n=1}^\infty||_p=
||x-y||_X+||y-z||_X.\] On the other hand, by the triangle
inequality in $\ell_\infty^n$,
\[\left\|\{||x_n-y_n||_\infty+||y_n-z_n||_\infty\}_{n=1}^\infty\right\|_p\ge
\left\|\{||x_n-z_n||_\infty\}_{n=1}^\infty\right\|_p=||x-z||_X.\]
This contradicts  \eqref{E:FlatTriple}.
\end{proof}

The next definition  will be used in the sequel.

\begin{definition}\label{D:MetricRay} A {\it metric ray} in a metric space $(A,d_A)$ is a
sequence $r=\{r_i\}_{i=0}^\infty$ of points such that the sequence
$d_A(r_i,r_0)$ is strictly increasing and, for $i<j<k$, the
following equality holds:
\begin{equation}\label{E:TriangleEq}
d_A(r_i,r_k)=d_A(r_i,r_j)+d_A(r_j,r_k).
\end{equation}
\end{definition}

For all of the metric rays in Banach spaces considered in this
paper, it will be  assumed that
\begin{equation}\label{E:Starts0}
r_0=0. \end{equation}

Consider subspaces
$X_k=\left(\oplus_{n=1}^k\ell_\infty^n\right)_p$ in $X$ and the
natural projections $P_k:X\to X_k$ defined by
$P(\{x_n\}_{n=1}^\infty)=\{x_n\}_{n=1}^k$.

\begin{lemma}\label{L:epsPart} For each metric ray $r=\{r_i\}_{i=0}^\infty$ in $X$ and each $\ep\in(0,1)$, there is
$k\in\mathbb{N}$ such that the natural projection $P_k:X\to X_k$
satisfies:
\begin{equation}\label{E:P_kToRay}
||P_kr_i-r_i||_X\le \ep ||r_i||_X \;\; for \;\;every
\;\;i=0,1,\dots
\end{equation}
 Under the assumption $r_0=0$, a number
$k$ satisfying this condition can be determined from the number
$\ep>0$ and the vector $r_1$.
\end{lemma}

\begin{proof} Let $r_i=\{r_{in}\}_{n=1}^\infty$, where
$r_{in}\in\ell_\infty^n$. With the help of Definition
\ref{D:MetricRay} and Lemma \ref{L:Multiple}, one derives that for
$i<j<k$, the vector
$\{||r_{jn}-r_{in}||_\infty\}_{n=1}^\infty\in\ell_p$ is a positive
multiple of $\{||r_{kn}-r_{jn}||_\infty\}_{n=1}^\infty$. Using the
fact that $r_{0n}=0$ for every $n$, it can be easily obtained that
any vector of the form $\{||r_{jn}-r_{in}||_\infty\}_{n=1}^\infty$
is a positive multiple of $\{||r_{1n}||_\infty\}_{n=1}^\infty$,
and any vector of the form $\{||r_{in}||_\infty\}_{n=1}^\infty$ is
also  a positive multiple of $\{||r_{1n}||_\infty\}_{n=1}^\infty$.
Now, pick $k\in\mathbb{N}$ such that
$||P_kr_1-r_1||_X\le\ep||r_1||_X$. This means that
$||\{||r_{1n}||_\infty\}_{n=k+1}^\infty||_p\le \ep
||\{||r_{1n}||_\infty\}_{n=1}^\infty||_p$. The fact that
$\{||r_{in}||_\infty\}_{n=1}^\infty$ is a positive multiple of
$\{||r_{1n}||_\infty\}_{n=1}^\infty$  leads to
$||\{||r_{in}||_\infty\}_{n=k+1}^\infty||_p\le \ep
||\{||r_{in}||_\infty\}_{n=1}^\infty||_p$, or $||P_kr_i-r_i||_X\le
\ep ||r_i||_X$, as required.
\end{proof}

In order to complete the proof of Theorem \ref{T:Below}, we
introduce a locally finite metric space $A$ which  does not admit
an isometric embedding into $X$.\medskip

To begin with, let $\{N_t\}_{t=1}^\infty$ be an increasing
sequence of positive integers so that
$\lim_{t\to\infty}N_t=\infty$. Consider the set
$S\subset\ell_\infty$ consisting of all sequences, for which the
first coordinate is a nonnegative integer, the next $N_1$
coordinates are nonnegative integer multiples of $3$, the next
$N_2$ coordinates are nonnegative integer multiples of $3^2$, the
next $N_3$ coordinates are nonnegative integer multiples of $3^3$,
and so on. Clearly, $S$ is countable. In addition, it is not
difficult to see that $S$ is locally finite implying that all of
its subsets are also locally finite.\medskip

Further, let $\{I_t\}_{t=0}^\infty$ be a partition of
$\mathbb{N}$, where $I_0=\{1\}$, $I_1=\{2,\dots,1+N_1\}$, and
$I_t=\{1+N_1+\dots+N_{t-1}+1,\dots,1+N_1+\dots+N_{t-1}+N_t\}$ for
$t\ge 2$. The definition of $S$ can be rewritten as: a sequence
$\{s_i\}_{i=1}^\infty\in\ell_\infty$ is in $S$ if and only if each
$s_i$ is a nonnegative integer multiple of $3^t$ for $i\in I_t$.
\medskip

Finally, a subset  $A\subset S$ is taken to be the union of metric
rays $r(j)$, $j\in\mathbb{N}$,  constructed as described below.
For each $j\in\mathbb{N}$ pick $n_1(j)\in I_1$, $n_2(j)\in I_2$,
etc. This can and will be performed in such a way that the next
condition is satisfied:

\begin{equation}\label{E:*}
\forall t\in \mathbb{N}\quad\forall n\in I_t\quad\exists
j\in\mathbb{N}\quad n=n_t(j).\end{equation}

After this, the collection $\{r(j)\}_{j=1}^\infty$ of metric rays,
where $r(j)=\{r_t(j)\}_{t=0}^\infty$, is defined as follows:

\begin{enumerate}

\item[(A)] $r_0(j)=0\in\ell_\infty$ (for every $j\in \mathbb{N}$).

\item[(B)] $r_1(j)$ is the unit vector
$(1,0,\dots,0,\dots)\in\ell_\infty$ (for every $j\in \mathbb{N}$).

\item[(C)] For $t\ge 2$, let $r_t(j)$ be the vector which has
$1+3+\dots+3^{t-1}$ as its first coordinate, $3+\dots+3^{t-1}$ as
its $n_1(j)$ coordinate, \dots, $3^{t-2}+3^{t-1}$ as its
$n_{t-2}(j)$ coordinate, $3^{t-1}$ as its $n_{t-1}(j)$ coordinate,
 while all the other coordinates are $0$.

\end{enumerate}

It can be noticed that each $r(j)$ is a metric ray and that, for
every $t$ and $j$, the vector $r_t(j)$ is in the set $S$ described
above.
\medskip

The set $A$ is locally finite, since it is a subset of $S$.
Suppose that $A$ admits an isometric embedding $E:A\to X$. Without
loss of generality, assume that $E(0)=0$ (recall that $0\in A$).
Clearly, isometries map metric rays onto metric rays. It will be
proved by applying Lemma \ref{L:epsPart} in the case where
$\ep\in(0,1)$ is sufficiently small,
 that the existence of such isometric embedding leads to a
contradiction.

Namely, select $\ep\in(0,1)$ in such a way that
\begin{equation}\label{E:ep}3^{t-1}-2\ep3^{t}\ge 3^{t-2}
\end{equation}
for every $t\in\mathbb{R}$. Here,  condition \eqref{E:ep} is
written in the form in which  it will be used. Applying Lemma
\ref{L:epsPart} to the ray $\{Er_t(j)\}_{t=0}^\infty$, we conclude
that there is $k\in\mathbb{N}$ such that
\begin{equation}\label{E:LemAppl}||P_kEr_t(j)-Er_t(j)||_X\le\ep||Er_t(j)||_X=\ep||r_t(j)||_\infty,\end{equation}
for every $t$, where the equality holds due to the fact that $E$
is an isometry mapping $0$ to $0$. The last statement of Lemma
\ref{L:epsPart} implies that $k$ depends only on the vector
$Er_1(j)$, and therefore does not depend on $j$ (by condition
(B)).

Set $m=\dim X_k$, where, as before,  $X_k=P_kX$. It is common
knowledge that there exists an absolute constant $C$ such that,
for any $\delta>0$, the cardinality of a $\delta$-separated set
inside a ball of radius $R$ in an $m$-dimensional Banach space
does not exceed $(CR/\delta)^m$. See \cite[Lemma
9.18]{Ost13}.\medskip

Denote by $B_t$ the ball of $A$ of radius $3^{t}$  centered at
$0$. Then  $P_kEB_t$ is contained in the ball of radius $3^{t}$ of
$X_k$. Hence, the mentioned  fact on $\delta$-separated sets
implies that the cardinality of a $3^{t-2}$-separated set in
$P_kEB_t$ does not exceed $(9C)^{m}$. By showing that the
construction of $A$ implies that $P_kEB_t$ contains a
$3^{t-2}$-separated set of cardinality $N_{t-1}$, one obtains a
contradiction, because $\{N_t\}_{t=1}^\infty$ is indefinitely
increasing.
\bigskip

To achieve  this goal, remark  that
 for any $t\in\mathbb{N}$, the vector $r_{t}(j)$ is in
$B_{t}$ and even in the ball of radius $1+3+3^2+\dots+3^{t-1}$.
Combining conditions \eqref{E:*} and (C), it is concluded that the
set of all vectors $\{r_{t}(j)\}_{j=1}^\infty$ contains a subset
of cardinality $N_{t-1}$ which is $3^{t-1}$-separated.

Applying inequality \eqref{E:LemAppl} to any two images
$Er_{t}(j_1)$ and $Er_{t}(j_2)$ of elements of this subset, what
follows can be reached:
\[\begin{split}||P_kEr_{t}&(j_1)-P_kEr_{t}(j_2)-(Er_{t}(j_1)-Er_{t}(j_2))||_X\\
&\le
||P_kEr_{t}(j_1)-Er_{t}(j_1)||_X+||P_kEr_{t}(j_2)-Er_{t}(j_2)||_X\\&\le\ep(||r_{t}(j_1)||_\infty+||r_{t}(j_2)||_\infty)\end{split}\]
and, as a result,
\[\begin{split}||P_kEr_{t}(j_1)&-P_kEr_{t}(j_2)||_X
\\
&\ge
||Er_{t}(j_1)-Er_{t}(j_2)||_X-\ep(||r_{t}(j_1)||_\infty+||r_{t}(j_2)||_\infty)
\\&
\ge 3^{t-1}-2\ep3^{t}\stackrel{\eqref{E:ep}}{\ge}
3^{t-2},\end{split}\] which confirms that $P_kEB_t$ contains a
$3^{t-2}$-separated set of cardinality $N_{t-1}$. This proves the
theorem.\end{proof}

\section{Proof of Theorem \ref{T:BL08Impr}}

\begin{proof} To prove Theorem \ref{T:BL08Impr} it suffices to
show that, given an $\ep>0$, every locally finite metric space
admits a bilipschitz embedding into $X$ with distortion $\le
(4+\ep)$.
\medskip

As in \cite{BL08}, we use the existence inside $X$ of a subspace
which is close to $\left(\oplus_{n=1}^\infty\ell_\infty^n\right)$,
where the direct sum is not an $\ell_p$-sum, but just a
finite-dimensional decomposition with small decomposition
constant. The existence of such a sum is derived from the
Maurey-Pisier theorem \cite{MP76} (see also
\cite[Theorems~2.55~and~2.56]{Ost13}) by the line of reasoning
which goes back to Mazur, see \cite[p.~4]{LT77}.\medskip

Since  our argument  is a modification of the one contained in
\cite{LT77}, the needed details of the  construction used there
are presented below  for the reader's convenience.

\begin{definition}\label{D:LamNorm}
Let $\lambda\in(0,1]$. A subspace $N\subset X^*$  is called {\it
$\lambda$-norming over a subspace $Y\subset X$} if
\[\forall y\in Y~\sup\{|f(y)|:~f\in N,~||f||\le 1\}\ge\lambda||y||.\]
\end{definition}

\begin{lemma}\label{L:LamNorm} For any $\lambda\in(0,1)$ and any finite-dimensional subspace $Y\subset X$ there exists
a finite-dimensional subspace $N\subset X^*$ which is
$\lambda$-norming over $Y$.
\end{lemma}

\begin{proof} The existence of such a  subspace can be established as follows. Let
$\{x_i\}_{i=1}^m$ be an $(1-\lambda)$-net in the unit sphere of
$Y$ and let $N$ be the linear span of functionals $x_i^*$
satisfying the conditions $||x_i^*||=1$ and $x_i^*(x_i)=1$. The
verification that $N$ is $\lambda$-norming is immediate.
\end{proof}

Let $\ep\in(0,1)$ and  $\{\ep_i\}_{i=1}^\infty$ be positive numbers
satisfying:
\begin{equation}\label{E:EpIneq}\prod_{i=1}^\infty(1-\ep_i)>1-\ep.\end{equation}
Denote by  $(M,d_M)$  the locally finite metric space which will
be embedded into $X$. Pick a point $O\in M$ and set:
\[M_n=\{x\in M:~d_M(x,O)\le R_n\},\]
where $\{R_n\}_{n=1}^\infty$ is the sequence defined in
\eqref{E:R12}--\eqref{E:ROdd2}. Let $c(n)$ be the cardinality of
$M_n$. As a consequence of Fr\'echet's observation, $M_n$ admits
an isometric embedding $E_n$ into $\ell^{c(n)}_\infty$. Further,
the Maurey-Pisier theorem states that the space $X$ contains a
subspace $Y_1$ such that there is a linear map
$S_1:Y_1\to\ell_\infty^{c(1)}$ satisfying
\[||y||\le ||S_1y||\le (1+\ep)||y||.\]
Consider a finite-dimensional subspace $N_1\subset X^*$ so that
$N_1$  is $(1-\ep_1)$-norming over $Y_1$ and set
\[W_1=(N_1)_{\top}:=\{x\in X:~\forall x^*\in N_1~~
x^*(x)=0\}.\] It is easy to derive from the definition of cotype
that $W_1$ has no nontrivial cotype. Applying the Maurey-Pisier
theorem once more, one finds a subspace $Y_2\subset W_1$ and a
linear map $S_2:Y_2\to\ell_\infty^{c(2)}$ satisfying
\[||y||\le ||S_2y||\le (1+\ep)||y||.\]
Now, take $N_2\subset X^*$ as a finite-dimensional subspace which
contains $N_1$ and is $(1-\ep_2)$-norming over $\lin(Y_1\cup
Y_2)$, and set $W_2=(N_2)_{\top}$.
\medskip

We continue in an obvious way. In the $n$-th step, we find a
subspace
\[Y_n\subset W_{n-1}=(N_{n-1})_{\top}\]
and a linear map $S_n:Y_n\to\ell_\infty^{c(n)}$ satisfying
\[||y||\le ||S_ny||\le (1+\ep)||y||.\]
It is clear that, for $u\in W_n$ and $v\in(N_n)_{\top}$, the
inequality below is true:
\begin{equation}\label{E:FDD}||u+v||\ge(1-\ep_n)||u||.\end{equation}
It is easy to see that $\{Y_i\}_{i=1}^\infty$ form a
finite-dimensional decomposition of the closed linear span of
$\bigcup_{i=1}^\infty Y_i=:Y$. Writing a sum of the form $\sum_{i=1}^\infty y_i$ we mean that $y_i\in Y_i$. We introduce the following norm on
$Y$:
\begin{equation}\label{E:NormY}
\left\|\sum_{i=1}^\infty
y_i\right\|_a=\max\left\{\left\|\sum_{i=1}^\infty
y_i\right\|_X,~~\max\{||S_jy_j||+||S_ky_k||:~~ j,k\in
\mathbb{N}\}\right\}.
\end{equation}
Let us show that the norm $||\cdot||_a$ is
$\displaystyle{\frac{4(1+\ep)}{1-\ep}}$\,-equivalent to
$||\cdot||_X$. In fact, it is clear that
\[\left\|\sum_{i=1}^\infty
y_i\right\|_X\le \left\|\sum_{i=1}^\infty y_i\right\|_a.\]
On the other hand, inequality  \eqref{E:FDD} yields:
\[(1-\ep_k)\left\|\sum_{i=1}^k y_i\right\|_X\le \left\|\sum_{i=1}^\infty
y_i\right\|_X\] and
\[(1-\ep_{k-1})\left\|\sum_{i=1}^{k-1} y_i\right\|_X\le \left\|\sum_{i=1}^\infty
y_i\right\|_X.\] By the triangle inequality,
\[||y_k||_X\le\left(\frac1{1-\ep_k}+\frac1{1-\ep_{k-1}}\right)\left\|\sum_{i=1}^\infty
y_i\right\|_X.\]  The stated above equivalence of $\|\cdot\|_a$
and $\|\cdot\|_X$ now follows from  $||S_ky_k||\le (1+\ep)||y_k||$
and \eqref{E:EpIneq}. \medskip

Observe that $\lin\{Y_j\cup Y_k\}$ with the norm $||\cdot||_a$ is
isometric to $\ell_\infty^{c(j)}\oplus_1 \ell_\infty^{c(k)}$.
Consider $M$ as a subset of $\ell_\infty$ such that $O\in M$
coincides with $0\in\ell_\infty$. This implies that the argument
used to prove Theorem \ref{T:Above} in the case $p=1$ can be
applied to get an embedding of distortion $\le (1+\ep)$ of $M$
into $(Y,||\cdot||_a)$. Indeed, let us define an embedding $T:M\to
Y$ by the formula \eqref{E:DefT} (we use $p=1$ in \eqref{E:ci2}
and \eqref{E:si2}). Now we can see that if $Tx$ and $Ty$ are in
the same sum of the form $\ell_\infty^{c(j)}\oplus_1
\ell_\infty^{c(k)}$, the desired estimate can be obtained in the
same way as in the final part of Section \ref{S:Between1and2}. On
the other hand, if $Tx$ and $Ty$ are not both in the same direct
sum of the form $\ell_\infty^{c(j)}\oplus_1 \ell_\infty^{c(k)}$,
then $||y||\le\ep||x||$. In this case the estimate also goes
through in exactly the same way as in
\eqref{E:EpEst1}--\eqref{E:EpEstFin}.

 To summarize,   an
embedding of $M$ into $(Y,||\cdot||_a)$ with distortion $\le
(1+\ep)$ exists. Combining this fact with the established above
equivalence between $||\cdot||_X$ and $||\cdot||_a$ on $Y$, one
obtains  an embedding into $X$ with distortion
$\le\displaystyle{\frac{4(1+\ep)^2}{1-\ep}}$. With $\ep\downarrow
0$, the result stated in Theorem \ref{T:BL08Impr} is proved.
\end{proof}

\section{An open problem}

In our opinion the most interesting open problem related to this
study is:

\begin{problem}\label{P:Large} Do there exist Banach spaces $X$ with $D(X)>1^+$?
\end{problem}

\section{Acknowledgements}

The second-named author gratefully acknowledges the support by
National Science Foundation DMS-1201269 and DMS-1700176 and by
Summer Support of Research program of St. John's University during
different stages of work on this paper.

We would like to thank the referee for careful reading of the
paper and for suggesting improvements of our presentation.

\end{large}

\renewcommand{\refname}{

\textsc{Department of Mathematics, Atilim University, 06830 Incek,
Ankara, TURKEY} \par \textit{E-mail address}:
\texttt{sofia.ostrovska@atilim.edu.tr}\par\medskip

\textsc{Department of Mathematics and Computer Science, St. John's
University, 8000 Utopia Parkway, Queens, NY 11439, USA} \par
  \textit{E-mail address}: \texttt{ostrovsm@stjohns.edu} \par

\end{document}